\documentclass[11pt]{amsart}

\usepackage{amsmath, amsthm, amssymb, amsfonts, enumerate}
\usepackage{enumitem}
\usepackage[colorlinks=true,linkcolor=blue,urlcolor=blue]{hyperref}
\usepackage{dsfont}
\usepackage{color}
\usepackage{geometry}
\usepackage{todonotes}
\usepackage{enumerate}
  
\newtheorem{theorem}{Theorem}[section]
\newtheorem{remark}[theorem]{Remark}

\newtheorem{lemma}[theorem]{Lemma}
\newtheorem{proposition}[theorem]{Proposition}
\newtheorem{corollary}[theorem]{Corollary}
\newtheorem{definition}[theorem]{Definition}

\theoremstyle{plain}

\newcommand{\BB}{\mathcal{B}}

\newcommand{\FF}{\mathcal{F}}

\newcommand{\PP}{\mathcal{P}}

\newcommand{\field}[1]{\mathbb{#1}}

\newcommand{\R}{\field{R}}
\newcommand{\N}{\field{N}}

\renewcommand{\P}{\field{P}}

\newcommand{\de}{\delta}

\newcommand{\la}{\lambda}

\newcommand{\si}{\sigma}

\newcommand{\ps}{\psi}

\newcommand{\Om}{\Omega}

\usepackage[utf8]{inputenc}
\usepackage[english]{babel}
\usepackage{indentfirst}
\usepackage{graphicx}
\usepackage{pdfpages}
\usepackage[title]{appendix}
\usepackage{fancyhdr}

\title[A note on stochastic dominance and compactness]{A note on stochastic dominance and compactness}
\author[M.~Nendel]{Max Nendel}
\address{Max Nendel, Center for Mathematical Economics, Bielefeld University, 33615 Bielefeld, Germany}
\email{\href{mailto:Max.Nendel@uni-bielefeld.de}{Max.Nendel@uni-bielefeld.de}}

\date{\today}

\numberwithin{equation}{section}

\begin{document}

\begin{abstract}
 In this work, we discuss completeness for the lattice orders of first and second order stochastic dominance. The main results state that, both, first and second order stochastic dominance induce Dedekind super complete lattices, i.e.~lattices in which every bounded nonempty subset has a countable subset with identical least upper bound and greatest lower bound. Moreover, we show that, if a suitably bounded set of probability measures is directed (e.g.~a lattice), then the supremum and infimum w.r.t.~first or second order stochastic dominance can be approximated by sequences in the weak topology or in the Wasserstein-$1$ topology, respectively. As a consequence, we are able to prove that a sublattice of probability measures is complete w.r.t.~first order stochastic dominance or second order stochastic dominance and increasing convex order if and only if it is compact in the weak topology or in the Wasserstein-$1$ topology, respectively. This complements a set of characterizations of tightness and uniform integrability, which are discussed in a preliminary section.

 
 % MSC Schluessel: 06B23 (Complete lattices, completions), 60E15 (Inequalities; stochastic orderings), 28A33 (Spaces of measures, convergence of measures), 60B10 (Convergence of probability measures)
\end{abstract}

\maketitle

\section{Introduction}

In this work, we discuss completeness for lattice orders arising from first and second order stochastic dominance, and their relation to tightness and uniform integrability, respectively. Given a lattice $L$, it is a well-known result, due to Birkhoff \cite[Section X.12, Theorem X.20]{MR0227053} and Frink \cite{MR6496}, that $L$ is \textit{complete}, i.e.~every nonempty subset of $L$ has a least upper bound and a greatest lower bound, if and only if $L$ is compact in the \textit{interval topology}. The latter is the smallest topology on $L$ such that all closed intervals of the form
$$(-\infty, a]:=\{x\in L \, |\, x\leq a\}\quad \text{and}\quad [a,\infty):=\{x\in L\, |\, x\geq a\}, \quad \text{for }a\in L,$$
are closed. Due to its definition, the interval topology or, more precisely, open sets in the interval topology are ususally not easy to describe. Moreover, in general, the interval topology is not even a Hausdorff topology. We refer to Baer \cite{MR0071400} for a characterization of lattices with Hausdorff interval topologies. Stochastic dominance or convex orders are present in many applications in microeconomics and decision theory (cf.~Levy \cite{MR3525602}), as well as mathematics, e.g.~martingale optimal transport (cf.~Strassen \cite{MR177430}) and, recently, submodular mean field games (cf.~Dianetti et al.~\cite{dffn19}). It is therefore desirable to obtain a more tractable characterization of complete (w.r.t.~stochastic dominance) lattices of probability measures than the one in terms of compactness in the interval topology. Two of the main results of this paper are a characterization of complete lattices w.r.t.~first and second order stochastic dominance in terms of compactness in the weak topology and in the Wasserstein-$1$ topology, respectively (Theorem \ref{complattice1order} and Theorem \ref{complattice2order}).\\

Dianetti et al.~\cite{dffn19} propose a lattice-theoretical approach to mean field games using Tarski's fixed point theorem. The crucial step in their analysis is to show that an appropriately chosen lattice $L$ of flows of probability measures is complete. Moreover, the approximation of suprema and infima is of fundamental importance in their analysis. Motivated by this application, in Section \ref{sec.completeness}, for a fixed $\sigma$-finite measure space $(S,\mathcal S,\pi)$, we consider the sets $L^0\big(S,\mathcal S,\pi;\PP(\R)\big)$ and $L^0\big(S,\mathcal S,\pi;\PP_1(\R)\big)$ of all equivalence classes of $\mathcal S$-$\BB\big(\PP(\R)\big)$ and $\mathcal S$-$\BB\big(\PP_1(\R)\big)$ measureable flows $(\mu_t)_{t\in S}$ of probability measures, respectively, and address the Dedekind completeness as well as the approximation of suprema and infima for the latter. Here, $\PP(\R)$ is the set of all probability measures on the Borel $\sigma$-algebra of $\R$ and $\PP_1(\R)$ is the subset of all probability measures with finite first moment, endowed with the lattice orders of first and second order stochastic dominance as well as the Borel $\sigma$-algebras $\BB\big(\PP(\R)\big)$ and $\BB\big(\PP_1(\R)\big)$ w.r.t.~the weak topology on $\PP(\R)$ and the Wasserstein-$1$ topology on $\PP_1(\R)$, respectively. Kertz and R\"osler \cite{MR1833858} proved that the lattices $\PP(\R)$ and $\PP_1(\R)$ are \textit{Dedekind complete}, i.e.~every bounded nonempty subset has a least upper bound and a greatest lower bound. In the present paper, we show that $L^0\big(S,\mathcal S,\pi;\PP(\R)\big)$ and appropriate sublattices of $L^0\big(S,\mathcal S,\pi;\PP_1(\R)\big)$ are \textit{Dedekind super complete}, i.e. every bounded nonempty subset has a countable subset with identical least upper bound and a greatest lower bound (Theorem \ref{thm.completeness} and Theorem \ref{thm.completeness2}). We further prove that, for directed and (suitably) bounded sets of probability measures, the supremum and infimum w.r.t.~first or second order stochastic dominance can be approximated by monotone sequences in the weak topology or the Wasserstein-$1$ topology, respectively. The proof relies on an abstract lemma (Lemma \ref{lem.auxres}) in the Appendix \ref{append.auxres}, which gives a sufficient and neccessary condition for the Dedekind super completeness of lattices. The proof of the latter is an abstract version the existence proof of the essential supremum for families of random variables (see e.g.~\cite[Theorem A.32]{MR2169807}). Choosing $S$ to be a singleton with $\pi(S)>0$, we obtain the Dedekind super completeness of $\PP(\R)$ and $\PP_1(\R)$, endowed with first order and second order stochastic dominance, as well as the approximation of suprema/infima in the weak topology and the Wasserstein-$1$ topology, respectively.\\

In Section \ref{sec.tightness}, we discuss some preliminary results on tightness and uniform integrability. Chandra \cite{MR3412766} gives a De La Vall\'ee Poussin-type characterization for tightness and uniform integrability in terms of a function $\psi\colon [0,\infty)\to [0,\infty)$ with a certain behaviour at infinity, see also Hu and Rosalsky~\cite{MR2740082}. Inspired by these two works, we relate properties of the function $\psi$, such as strict monotonicity and (strict) convexity, to integrability conditions on the set of distributions. In particular, we answer two open questions from \cite{MR3412766}, and derive a perturbation result for families of uniformly integrable random variables in terms of a strictly convex transformation (Lemma \ref{corui1} and Corollary \ref{corui2}). More precisely, we show that, for a family $H$ of uniformly integrable random variables on a probability space $(\Om,\FF,\P)$, there exists a strictly convex nondecreasing function $\psi\colon [0,\infty)\to [0,\infty)$ with $\tfrac{\psi(s)}{s}\to \infty$ as $s\to \infty$, such that $\{\psi(|X|)\, |\, X\in H\}$ is again u.i. Leskel\"a and Vhiola \cite{MR2998767} obtained several characterizations of tightness and uniform integrability in terms of first and second order stochastic dominance for measures on the positive half line. We extend the results by Leskel\"a and Vhiola \cite{MR2998767} to $\R$ and combine them with the results from Section \ref{sec.tightness} and Section \ref{sec.completeness} in order to obtain a characterization of tightness and uniform integrability in terms of integrability conditions for a function $\psi$ with certain properties and boundedness conditions w.r.t.~first and second order stochastic dominance, respectively (Lemma \ref{equivtight} and Lemma \ref{equivui}).\\
 
\noindent {\bf Structure of the paper:} The paper is organized as follows. In Section \ref{sec.tightness}, we derive some preliminary results on tightness and uniform integrability, which we could not find in the literature in this form. Section \ref{sec.completeness} is dedicated to the completeness of lattices of (flows of) probability measures ordered by first and second order stochastic dominance. The main results are Theorem \ref{thm.completeness}, Theorem \ref{thm.completeness2}, Theorem \ref{complattice1order} and Theorem \ref{complattice2order}. In the Appendix \ref{append.auxres}, we first give some lattice-theoretical definitions (Definition \ref{def.complete} and Definition \ref{def:increasing}), and then derive an abstract auxiliary result (Lemma \ref{lem.auxres}) that helps to determine, when a lattice is Dedekind super complete, and forms the basis for the proofs in Section \ref{sec.completeness}. The results from Section \ref{sec.tightness} are proved in the Appendix \ref{appendB}.

\section{Some preliminary remarks on tightness and uniform integrability}\label{sec.tightness}

Let $\R_+:=[0,\infty)$ and $\mathcal B(\R_+)$ denote the Borel $\sigma$-algebra on $\R_+$. We first concentrate on the set $\mathcal M(\R_+)$ of all (possibly nonfinite) measures $\mu\colon\mathcal \mathcal B(\R_+)\to [0,\infty]$. We start with the following characterization of tightness, which is sort of a folk theorem, see e.g. \cite[Lemma 2.1]{MR3412766}, \cite[Remark 7]{MR2998767} or \cite[Lemma D.5.3]{MR2509253}. However, in this form, we could not find it in the literature, which is why we would like to provide a proof in the Appendix \ref{appendB}.

\begin{lemma}\label{firstcrittight}
  Let $K\subset \mathcal M(\R_+)$. Then, the following statements are equivalent:
 \begin{enumerate}
  \item[(i)] $K$ is tight, i.e. $\sup_{\nu\in K}\nu\big((s,\infty)\big)\to 0$ as $s\to \infty$.
  \item[(ii)] There exists a nondecreasing function $\psi\colon [0,\infty)\to [0,\infty)$ with $\psi(0)=0$, $\psi(s)\to \infty$ as $s\to \infty$ and
 \[
  \sup_{\nu\in K}\int_0^\infty \psi(s)\, {\rm d}\nu(s)<\infty.
 \]
 \end{enumerate}
 The function $\psi$ in (ii) can be chosen to be continuous. The function $\psi$ can be chosen to be convex if and only if
 \begin{equation}\label{convexcrit}
  \sup_{\nu\in K}\int_0^\infty (s-M)^+\, {\rm d}\nu(s) <\infty\quad \text{or, equivalently, if}\quad \sup_{\nu\in K}\int_M^\infty s\, {\rm d}\nu(s) <\infty
 \end{equation}
 for some $M\geq 0$. Additionally, the function $\psi$ can be chosen to be strictly increasing if and only if
 \begin{equation}\label{convexcrit1}
  \sup_{\nu\in K}\nu\big((s,\infty)\big)<\infty \quad \text{for all }s>0.
 \end{equation}
\end{lemma}

 Lemma \ref{firstcrittight} answers an open question by an anonymous referee concerning Lemma 2.1 in \cite{MR3412766}, namely if the function $\psi$ in (ii) can be chosen to be convex. In the previous lemma, we have seen that $\psi$ can be chosen to be convex if and only if \eqref{convexcrit} is satisfied for some $M\geq 0$. In the special case, where $K$ is a set of probability measures, \eqref{convexcrit} is equivalent to the uniform boundedness of the first moments of the elements of $K$, i.e.
 \[
  \sup_{\nu\in K} \int_0^\infty s\, {\rm d}\nu(s)<\infty.
 \]

\begin{definition}
Let $K\subset \mathcal M(\R_+)$. Then, we say that a nondecreasing map $\varphi\colon[0,\infty)\to [0,\infty)$ is \textit{uniformly integrable (u.i)} for $K$ if
 \[
  \sup_{\nu\in K}\int_M^\infty \varphi(s)\, {\rm d}\nu(s)\to 0 \quad \text{as }M\to \infty.
 \]
\end{definition}

We would like to mention at this point that the name ``uniform integrability'' is actually slightly misleading, since, for sets of nonfinite measures, uniform integrability does not imply integrability, not even for singletons.

\begin{remark}\label{remuni1}
 Let $\varphi\colon[0,\infty)\to [0,\infty)$ be nondecreasing. Then, the uniform integrability of $\varphi$ for a set $K\subset \mathcal M(\R_+)$ can be divided into three cases:
 \begin{enumerate}
  \item[(i)] $\varphi(s)=0$ for all $s\geq 0$. Then, $\varphi$ is always u.i. for $K$.
  \item[(ii)] $\varphi$ is bounded and $\varphi(s)\neq 0$ for some $s\geq 0$. Then, $\varphi$ is u.i. for $K$ if and only if $K$ is tight.
  \item[(iii)] $\varphi(s)\to \infty$ as $s\to \infty$. Then, $\varphi$ is u.i. for $K$ if and only if the identity $[0,\infty)\to [0,\infty),\;s\mapsto s$ is u.i. for the set $\{\mu\circ\varphi^{-1}\, |\, \mu\in K\}$.
 \end{enumerate}
 Hence, except in the trivial case (i), the question of uniform integrability of $\varphi$ can be reduced to the question of tightness, which has already been discussed in Lemma \ref{firstcrittight} or uniform integrability of the identity. Therefore, in the next lemma, which is a generalized version of the De La Vall\'{e}e Poussin Lemma (see also \cite{MR3412766}), we will just discuss the uniform integrability of the identity. The proof is a direct consequence of Lemma \ref{firstcrittight}, and is also relegated to the Appendix \ref{appendB}.
\end{remark}

In the following, we make use of the convention $\tfrac{0}{0}:=0$.

\begin{lemma}\label{corui1}
 Let $K\subset \mathcal M(\R_+)$. Then, the following statements are equivalent:
 \begin{enumerate}
  \item[(i)] The identity $[0,\infty)\to [0,\infty),\;s\mapsto s$ is uniformly integrable for $K$,
  \item[(ii)] There exists a nondecreasing function $\psi\colon [0,\infty)\to [0,\infty)$ with $\psi(0)=0$ such that $[0,\infty)\to [0,\infty),\; s\mapsto \tfrac{\psi(s)}{s}$ is nondecreasing, $\tfrac{\psi(s)}{s}\to \infty$ as $s\to \infty$ and
 \[
  \sup_{\nu\in K}\int_0^\infty \psi(s)\, {\rm d}\nu(s)<\infty.
 \]
 \end{enumerate}
 The function $\psi$ in (ii) can be chosen to be continuously differentiable, convex and u.i. for $K$. Additionally, the function $\psi$ can be chosen to be strictly convex (and thus strictly increasing) if and only if \eqref{convexcrit1} is satisfied.
\end{lemma}

 Notice that in the previous lemma, the function $\psi$ in (ii) can be chosen to be strictly convex and u.i. if $K$ is, for example, a set of probability measures. This is actually quite remarkable, since it shows that uniform integrability is preserved under a ``small'' perturbation in terms of a strictly convex transformation $\psi$ with $\frac{\psi(s)}{s}\to \infty$ as $s\to \infty$. It can thus be seen as a perturbation result for families of uniformly integrable functions and answers an open question by an anonymous referee concerning Lemma 1.2 in \cite{MR3412766}. In view of the standard application of the De La Vall\'{e}e Poussin Lemma, where the function $\psi$ is chosen to be $\psi(s)=s^p$ for some suitable $p>1$, this perturbation result is quite natural, since, in this case, any perturbation of the form $s\mapsto s^q$ for $q\in (1,p)$ would remain u.i.

\begin{corollary}\label{corui2}
 Let $H$ be a set of random variables or random vectors (not necesarily of the same dimension) on a probability space $(\Omega,\mathcal F,\mathbb P)$ and $p\in [0,\infty)$. Then, the following statements are equivalent:
 \begin{enumerate}
  \item[(i)] $H$ is uniformly integrable (u.i.), i.e.
  \[
   \sup_{X\in H}\int_{\{|X|> s\}} |X|\, {\rm d}\mathbb P\to 0 \quad \text{as }s\to \infty.
  \]
  \item[(ii)] There exists a nondecreasing function $\psi\colon [0,\infty)\to [0,\infty)$ with $\psi(0)=0$, $\tfrac{\psi(s)}{s}\to \infty$ as $s\to \infty$ and
  \[
   \sup_{X\in H}\int_\Omega \psi(|X|)\, {\rm d}\mathbb P<\infty.
  \]
 \end{enumerate}
 The function $\psi$ in (ii) can be chosen to be continuously differentiable, strictly convex and in such a way that $\big\{\psi(|X|)\, \big|\, X\in H\big\}$ is again u.i.
\end{corollary}

\begin{proof}
 This follows from the previous results by considering $K:=\big\{\mathbb P\circ |X|^{-1}\, \big|\, X\in H\big\}$.
\end{proof}

\section{Completeness}\label{sec.completeness}

In this section, we discuss the completeness of lattice orders arising from first and second order stochastic dominance. Throughout this section, let $\BB(\R)$ denote the Borel $\sigma$-algebra on $\R$, $\PP(\R)$ denote the set of all probability measures on $\BB(\R)$ and $\PP_1(\R)$ denote the set of all probability measures on $\BB(\R)$ with finite first moment, i.e. all probability measures $\mu$ on $\BB(\R)$ with
\[
 \int_\R |x|\, {\rm d}\mu(x)<\infty.
\]
\subsection{First order stochastic dominance}

Throughout this section, we will identify a distribution $\mu\in \PP(\R)$ by its survival function $\mu_0$, i.e. we identify
\[
 \mu(s)=\mu_0(s):=\mu\big((s,\infty)\big)
\]
for all $s\in \R$. On $\PP(\R)$, we consider the partial order $\leq_{\rm st}$ arising from \textit{first order stochastic dominance}, given by
\[
 \mu\leq_{\rm st}\nu \quad \text{if and only if}\quad \mu_0(s)\leq \nu_0(s)\quad \text{for all }s\in \R.
\]
Recall that, for $\mu,\nu\in \PP(\R)$, $\mu\leq_{\rm st}\nu$ if and only if
\[
 \int_\R h(x)\, {\rm d}\mu (x)\leq \int_\R h(x)\, {\rm d}\nu(x)
\]
for all nondecreasing functions $h\colon \R\to \R$. For a detailed discussion on the properties of the partial order $\leq_{\rm st}$, we refer to \cite[Section 1.A]{MR2265633}.

\begin{remark}\label{minimaandmaxima}\
\begin{enumerate}
 \item[a)] By identifying $\mu$ with its survival function $\mu_0$, the set $\mathcal P(\R)$ coincides with the set of all nonincreasing right-continuous functions $F\colon \R\to [0,1]$ with $\lim_{s\to -\infty}F(s)=1$ and $\lim_{s\to \infty} F(s)=0$. In particular, the partial order $\leq_{\rm st}$ induces a lattice structure on $\PP(\R)$ via
 \[
  \qquad \quad \big(\mu\vee_{\rm st}\nu\big)(s):=\mu_0(s)\vee \nu_0(s)\quad \text{and}\quad \big(\mu\wedge_{\rm st}\nu\big)(s):=\mu_0(s)\wedge \nu_0(s)\quad \text{for all }s\in \R.
 \]
 Further, we would like to recall that the weak topology on $\PP(\R)$ is metrizable and that the weak convergence coincides with the pointwise convergence of survival functions at every continuity point, i.e. $\mu^n\to \mu$ weakly as $n\to \infty$ if and only if
 \[
  \qquad\mu_0^n(s)\to \mu_0(s)\quad\text{as }n\to \infty\quad \text{for every continuity point }s\in \R\text{ of }\mu_0.
 \]
 Therefore, the weak convergence behaves well with the pointwise lattice operations $\vee_{\rm st}$ and $\wedge_{\rm st}$. More precisely, the maps $(\mu,\nu)\mapsto \mu \vee_{\rm st} \nu $ and $(\mu,\nu)\mapsto \mu \wedge_{\rm st} \nu $ are continuous $\mathcal P(\R)\times \mathcal P(\R)\to \mathcal P(\R)$. Moreover, the weak topology is finer than the interval topology, since, by the previous argumentation, every closed interval is weakly closed.
 \item[b)] Recall that a nonincreasing function $\R\to \R$ is right-continuous if and only if it is lower semicontinuous (lsc). Hence, for a sequence $(\mu^n)_{n\in \N}\in \mathcal P(\R)$, which is bounded above, the supremum $\sup_{n\in \N}\mu^n$ w.r.t.~$\leq_{\rm st}$ exists, and is exactly the pointwise supremum of the survival functions $(\mu_0^n)_{n \in \mathbb{N}}$.
 \item[c)] For a nonincreasing function $F\colon \R\to \R$, we define its lsc-envelope $F_*\colon \R\to \R$ by
 \[
  F_*(s):=\sup_{\delta >0}F(s+\delta)\quad \text{for }s\in \R
 \]
 Notice that $F(s)\geq F_*(s)\geq F(s+\varepsilon)$ for all $s\in \R$ and $\varepsilon >0$. That is, $F_*$ differs from $F$ only at discontinuity points of $F$. Intuitively speaking, $F_*$ is the right continuous version of $F$. For a sequence $(\mu^n)_{n\in \N}\in \mathcal P(\R)$, which is bounded below, the infimum $\inf_{n\in \N}\mu^n$ w.r.t.~$\leq_{\rm st}$  exists, and is given by the lsc-envelope of the pointwise infimum of the survival functions $(\mu_0^n)_{n\in \N}$. That is, one has to modify the pointwise infimum at all its discontinuity points in order to be right continuous.
\item[d)] A combination of the previous remarks, leads to the following insight: Every bounded and nondecreasing or nonincreasing sequence $(\mu^n)_{n\in \N}\subset \mathcal P(\R)$ converges weakly to its supremum or infimum w.r.t.~$\leq_{\rm st}$, respectively.
\end{enumerate}
\end{remark}

Let $(S,\mathcal S,\pi)$ be a $\sigma$-finite measure space. We denote the Borel $\sigma$-algebra of the weak topology by $\mathcal B\big(\PP(\R)\big)$ and the lattice of all equivalence classes of $\mathcal S$-$\mathcal B\big(\PP(\R)\big)$-measurable functions $S\to \mathcal P(\R)$ by $L_{\rm st}^0=L^0\big(\PP(\R)\big)=L^0\big(S,\mathcal S,\pi;\PP(\R)\big)$. An arbitrary element $\mu$ of $L_{\rm st}^0$ will be denoted in the form $\mu=(\mu_t)_{t\in S}$. On $L_{\rm st}^0$ we consider the order relation $\leq_{L_{\rm st}^0}$, given by $\mu \leq_{L_{\rm st}^0} \nu$ if and only if $\mu_t \leq_{\rm st} \nu_t$ for $\pi$-a.a.\ $t\in S$.

\begin{theorem}\label{thm.completeness}\
 \begin{enumerate}
  \item[a)] The lattice $L_{\rm st}^0$ is Dedekind super complete (cf. Definition \ref{def.complete}).  
  \item[b)] If $M\subset L_{\rm st}^0$ is a nonempty set, which is bounded above or below and directed upwards or downwards (cf. Definition \ref{def.complete}), then there exists a sequence $(\mu^n)_{n\in \N}\subset M$ with $\mu^n\leq_{L_{\rm st}^0} \mu^{n+1}$ for all $n\in \N$ and $\mu^n\to \sup M$ weakly $\pi$-a.e.\ as $n\to \infty$ or $\mu^n\geq_{L_{\rm st}^0} \mu^{n+1}$ for all $n\in \N$ and $\mu^n\to \inf M$ weakly $\pi$-a.e.\ as $n\to \infty$, respectively.
  %  \item[b)] Let $M\subset L_{\rm st}^0$ be nonempty. If $M$ is bounded above and directed upwards (cf. Definition \ref{def.complete}), then there exists a sequence $(\mu^n)_{n\in \N}\subset M$ with $\mu^n\leq_{L_{\rm st}^0} \mu^{n+1}$ for all $n\in \N$ and $\mu^n\to \sup M$ weakly $\pi$-a.e.\ as $n\to \infty$. Analogously, if $M$ is bounded below and directed downwards, then there exists a sequence $(\mu^n)_{n\in \N}\subset M$ with $\mu^n\geq_{L_{\rm st}^0} \mu^{n+1}$ for all $n\in \N$ and $\mu^n\to \inf M$ weakly $\pi$-a.e.\ as $n\to \infty$.
  %Analogously, if $M$ is bounded below and directed downwards, then there exists a sequence $(\mu^n)_{n\in \N}\subset M$ with $\mu^n\geq_{L_{\rm st}^0} \mu^{n+1}$ for all $n\in \N$ and $\mu^n\to \inf M$ weakly $\pi$-a.e.\ as $n\to \infty$.
 \end{enumerate}
\end{theorem}

\begin{proof}
 Since every $\sigma$-finite measure can be transformed into a finite measure without changing the null-sets, we may w.l.o.g. assume that $\pi$ is finite. By Remark \ref{minimaandmaxima}, $L_{\rm st}^0$ is Dedekind $\sigma$-complete. Let $\Phi$ be the cumulative distribution function of the standard normal distribution, i.e.
 \[
  \Phi(x):=\frac{1}{\sqrt{2\pi}}\int_{-\infty}^x e^{-y^2/2}\, {\rm d}y\quad\text{for all }x\in \R.
 \]
 Alternatively, one could consider e.g. $\Phi=\arctan$. The map $S\to \R, \; t\mapsto \int_\R\Phi(x)\, {\rm d}\mu_t(x)$ is $\mathcal S$-$\BB(\R)$-measurable for every $\mu\in L_{\rm st}^0$, since $\Phi\in C_b(\R)$ induces a continuous (w.r.t.~the weak topology) linear functional $\mathcal P(\R)\to \R$. Hence,
 \[
  F\colon L_{\rm st}^0\to \R,\quad \mu\mapsto \int_S\int_\R\Phi(x)\, {\rm d}\mu_t(x)\,{\rm d}\pi(t)
 \]
 is well-defined and strictly increasing (cf. Definition \ref{def:increasing}), since $\Phi$ is strictly increasing, see e.g. \cite[Theorem 1.A.8]{MR2265633}. The assertions now follow from Lemma \ref{lem.auxres} with $R=\R$, Remark \ref{rem.auxres} and Remark \ref{minimaandmaxima} d).
\end{proof}

In the case, where $S$ is a singleton and $\pi(S)>0$, we obtain the following corollary.

\begin{corollary}\label{cor.completeness}\
 \begin{enumerate}
  \item[a)] The lattice $\PP(\R)$, endowed with the lattice order $\leq_{\rm st}$, is Dedekind super complete.
  \item[b)] If $M\subset \PP(\R)$ is a nonempty set, which is bounded above or below and directed upwards or downwards, then there exists a sequence $(\mu^n)_{n\in \N}\subset M$ with $\mu^n\leq_{\rm st} \mu^{n+1}$ for all $n\in \N$ and $\mu^n\to \sup M$ weakly as $n\to \infty$ or $\mu^n\geq_{\rm st} \mu^{n+1}$ for all $n\in \N$ and $\mu^n\to \inf M$ weakly as $n\to \infty$.
  %Analogously, if $M$ is bounded below and directed downwards, then there exists a sequence $(\mu^n)_{n\in \N}\subset M$ .
 \end{enumerate}
\end{corollary}

We now turn our focus on characterizing the complete sublattices of $\PP(\R)$. We start with the following lemma, which gives two more characterizations of tightness. It is basically a combination and generalization of \cite[Proposition 1]{MR2998767} and Lemma \ref{firstcrittight}.

\begin{lemma}\label{equivtight}
 Let $K\subset \mathcal P(\R)$. Then, the following statements are equivalent:
 \begin{enumerate}
  \item[(i)] $K$ is tight, i.e.~$\sup_{\mu\in K}\mu\big([-s,s]^c\big)\to 0$ as $s\to \infty$.
  \item[(ii)] $K$ is $\leq_{\rm st}$-bounded, i.e.~there exist $\underline\mu,\overline\mu\in \PP(\R)$ with $\underline\mu\leq_{\rm st} \mu\leq_{\rm st} \overline\mu$ for all $\mu\in K$.
  \item[(iii)] Every nonempty subset of $K$ has a least upper bound and a greatest lower bound.
  \item[(iv)] There exists a nondecreasing function $\psi\colon [0,\infty)\to [0,\infty)$ with $\psi(0)=0$, $\psi(s)\to \infty$ as $s\to \infty$ and
  \[
   \sup_{\mu\in K}\int_\R \psi(|x|) \, {\rm d}\mu(x)<\infty.
  \]
 \end{enumerate}
 The function $\psi$ in (iv) can be chosen to be continuous and strictly increasing. Moreover, the function $\psi$ can be chosen to be convex if and only if
 \[
  \sup_{\mu\in K}\int_\R |x|\, {\rm d}\mu(x) <\infty.
 \]
\end{lemma}

\begin{proof}
 The equivalence of (i) and (iv) has been discussed in the previous section. By Corollary \ref{cor.completeness}, (ii) and (iii) are equivalent. If $K$ is $\leq_{\rm st}$-bounded, then
 \[
  \sup_{\mu\in K}\mu\big((s,\infty)\big)\leq \overline\mu_0(s)\to 0\quad \text{as }s\to \infty
 \]
 and
 \[
  \sup_{\mu\in K}\mu\big((-\infty,s)\big)=1-\inf_{\mu\in K}\mu\big([s,\infty)\big)\leq 1-\underline\mu_0(s)\to 0\quad \text{as }s\to -\infty,
 \]
 which shows that $K$ is tight. It remains to show that (i) implies (ii). We assume that $K$ is tight. Then,
 \begin{equation}\label{eq.tightminmax}
  \inf_{\mu\in K}\mu_0(s)\to 1\quad\text{as } s\to -\infty\quad \text{and}\quad \sup_{\mu\in K} \mu_0(s)\to 0\quad  \text{as }s\to \infty.
 \end{equation}
 We define $\underline\mu_0,\overline\mu_0\colon \mathbb{R} \rightarrow [0,1]$ by 
 \[
  \underline\mu_0(s):=\sup_{\delta>0}\inf_{\mu\in K}\mu_0(s+\delta) \quad \text{and} \quad \overline\mu_0(s):=\sup_{\mu\in K}\mu_0(s)
 \]
for $s\in \R$. Then, $\underline\mu_0$ and $\overline\mu_0$ are nonincreasing and right-continuous functions, cf. Remark \ref{minimaandmaxima}, and, by \eqref{eq.tightminmax}, $\lim_{s\to -\infty}F(s)=1$ and $\lim_{s\to \infty} F(s)=0$ for $F=\underline\mu_0,\overline\mu_0$. This shows that $\underline\mu_0$ and $\overline\mu_0$ give rise to two probability measures $\underline\mu,\overline\mu\in \mathcal P(\R)$ with $\underline\mu\leq_{\rm st} \mu\leq_{\rm st} \overline\mu$ for all $\mu\in K$.
\end{proof}

%\begin{remark}\label{rem.bddtight}
% Let $K\subset \mathcal P(\R)$ be bounded, i.e.~there exist $\underline\mu,\overline\mu\in \PP(\R)$ with $\underline\mu\leq_{\rm st} \mu\leq_{\rm st} \overline\mu$ for all $\mu\in K$. Then, $K$ is tight, i.e. $\sup_{\mu\in K}\mu\big([-s,s]^c\big)\to 0$ as $s\to \infty$. In fact,
% \[
%  \sup_{\mu\in K}\mu\big((s,\infty)\big)\leq \mu^{\rm Max}(s)\to 0\quad \text{as }s\to \infty
% \]
% and
% \[
%  \sup_{\mu\in K}\mu\big((-\infty,s)\big)=1-\inf_{\mu\in K}\mu\big([s,\infty)\big)\leq 1-\mu^{\rm Min}(s+1)\to 0\quad \text{as }s\to -\infty
% \]
%\end{remark}

\begin{theorem}\label{complattice1order}
 Let $K$ be a sublattice of $\mathcal P(\R)$. Then, the following statements are equivalent:
 \begin{enumerate}
  \item[(ii)] $K$ is weakly compact.
  \item[(i)] $K$ is complete w.r.t.~the lattice order $\leq_{\rm st}$.
 \end{enumerate}
\end{theorem}

\begin{proof}
 Since, by Remark \ref{minimaandmaxima} a), the weak topology is finer than the interval topology, it follows that every weakly compact subset of $\mathcal P(\R)$ is compact in the interval topology, and therefore complete. Now, assume that $K$ be complete. We first show that $K$ is weakly closed. Let $(\mu^n)_{n\in \N}\subset \mathcal P(\R)$ with $\mu^n\to \mu\in \mathcal P(\R)$ weakly as $n\to \infty$. Then,
 \[
  \mu=\inf_{n\in \N}\sup_{k\geq n} \mu^k\in K,
 \]
 since $K$ is complete. Here, we used the fact that the weak convergence is equivalent to the pointwise convergence at every continuity point. Since the weak topology is metrizable, it follows that $K$ is weakly closed. By the previous lemma, $K$ is tight, since it is complete and thus $\leq_{\rm st}$-bounded. Since $K$ is weakly closed, it is therefore weakly compact by Prokhorov's theorem (see e.g. \cite[Theorem 8.6.2]{MR2267655} or \cite[Theorems 6.1 and 6.2]{MR1700749}).  
\end{proof}

%\begin{remark}\label{remintervalui}
% Consider the situation of the previous lemma. Then, $\mu^{\rm Min}$ and $\mu^{\rm Max}$ can be expressed in terms of $\psi$. Assume that $\psi$ is continuous and strictly increasing. Let $C>1$ with
% \[
% \sup_{\mu\in K}\int_\R |x| d\mu(x) \leq C.
% \]
%We extend $\psi$ to $(-\infty,0)$ by $\psi(s)=0$ for $s<0$. Then, we define
%\[
%\mu^{\rm Min}(s):=\frac{C}{\psi(-s)}\wedge 1 \quad \text{and} \quad \mu^{\rm Max}(s):=\bigg(1-\frac{C}{\psi(s)}\bigg)\vee 0.
%\]
%Moreover, the map $ x\mapsto \psi(|x|)^\alpha$ is u.i for $[\mu^{\rm Min},\mu^{\rm Max}]$ for all $0\leq \alpha<1$. In fact,
%\[
% \int_\R \psi(|x|)^\alpha\, {\rm d}\mu\leq \int_0^\infty \psi(s)^\alpha\, {\rm d}\mu^{\rm Max}+\int_0^\infty \psi(-s)^\alpha\, {\rm d}\mu^{\rm Min}.
%\]
%It follows that 
%\[
%  \int_0^\infty \psi(s)^\alpha\, {\rm d}\mu^{\rm Max}=\int_0^1 \bigg(\frac{C}{1-u}\bigg)^\alpha\, {\rm d}u<\infty
%\]
%and
%\[
% \int_0^\infty \psi(-s)^\alpha\, {\rm d}\mu^{\rm Min}=\int_0^1 \bigg(\frac{C}{u}\bigg)^\alpha\, {\rm d}u<\infty.
%\]
%By the De La Vall\'{e}e-Poussin Lemma, it follows that $\psi^\alpha$ is u.i. for $[\mu^{\rm Min},\mu^{\rm Max}]$. In particular, if $\psi(s)\geq s^p$ for some $p\in (0,\infty)$, %then, $x\mapsto |x|^q$ is u.i. for $[\mu^{\rm Min},\mu^{\rm Max}]$ for all $q\in (0,p)$.
%\end{remark}

\subsection{Second order stochastic dominance and increasing convex order}

In this section, we identify a probability measure $\mu\in \PP_1(\R)$ via its \textit{(negative) integrated distribution function} $\mu_{1,-}$, which is given by
 \[
  \mu_{1,-}(s)=\int_\R (x-s)^-\, {\rm d}\mu(x)=-\int_{-\infty}^s \mu(u)\, {\rm d}u\quad \text{for all } s\in \R,
 \]
 and via its \textit{integrated survival function} $\mu_{1,+}$, which is given by
 \[
  \mu_{1,+}(s)=\int_\R (x-s)^+\, {\rm d}\mu(x)=\int_s^\infty \mu_0(u)\, {\rm d}u\quad \text{for all } s\in \R.
 \]
 On $\PP_1(\R)$, we then consider the partial order arising from \textit{second order stochastic dominance} (also called \textit{increasing concave order}) $\leq_{icv}$, given by
 \[
  \mu\leq_{\rm icv}\nu \quad \text{if and only if}\quad \mu_{1,-}(s)\leq \nu_{1,-}(s)\quad \text{for all }s\in \R.
 \]
 and the \textit{increasing convex order} $\leq_{\rm icx}$, given by
\[
 \mu\leq_{\rm icx}\nu \quad \text{if and only if}\quad \mu_{1,+}(s)\leq \nu_{1,+}1(s)\quad \text{for all }s\in \R.
\]
 Recall that, for $\mu,\nu\in \PP_1(\R)$, $\mu\leq_{\rm icv}\nu$ ($\mu\leq_{\rm icx}\nu$) if and only if
 \begin{equation}\label{cxfcnt}
  \int_\R h(x)\, {\rm d}\mu (x)\leq \int_\R h(x)\, {\rm d}\nu(x)
 \end{equation}
 for all nondecreasing concave (convex) functions $h\colon \R\to \R$. In particular, $\mu\leq_{\rm st} \nu$ implies $\mu\leq_{\rm icv}\nu$ and $\mu\leq_{\rm icx}\nu$. If \eqref{cxfcnt} holds for all convex functions $h\colon \R\to \R$, we write $\mu\leq_{\rm cx} \nu$ (\textit{convex order}). Notice that $\mu\leq_{\rm cx} \nu$ if and only if $\nu\leq_{\rm icv} \mu$ and $\mu\leq_{\rm icx} \nu$. For $\mu\in \PP_1(\R)$, let $\mu'\in \PP_1(\R)$ denote the probability measure, given by
 \[
  \mu'\big((-\infty,s]\big):=\mu\big([-s,\infty)\big) \quad \text{for all }s\in \R.
 \]
 That is, for a random variable $X$ on a probability space $(\Om,\FF,\P)$ with distribution $\mu$, $\mu'$ is the distribution of $-X$. Then, for all $\mu,\nu\in \PP_1(\R)$,
 \begin{equation}\label{icxicv}
  \mu\leq_{\rm icx} \nu \quad \text{if and only if}\quad \mu'\geq_{\rm icv}\nu'.
 \end{equation}
 %For this reason, we will prove some results in this subsection only for the increasing concave order, since they directly carry over to the increasing convex order via \eqref{icxicv}.
 For a detailed discussion on the properties of the partial orders $\leq_{\rm icv}$ and $\leq_{\rm icx}$, we refer to \cite[Section 4.A]{MR2265633}. In this subsection, we wil investigate the interplay between the partial orders $\leq_{\rm icv}$ and $\leq_{\rm icx}$ and the Wasserstein-$1$ topology on the Wasserstein-$1$ space $\PP_1(\R)$. For a definition of the Wasserstein-$1$ metric and a detailed discussion on its properties, we refer to \cite{MR2459454}.
\begin{remark}\label{minimaandmaxima2ndorder}\
 \begin{enumerate}
  \item[a)] Notice that by identifying $\mu$ via its integrated distribution (survival) function $\mu_{1,-}$ ($\mu_{1,+}$), $\mathcal P_1(\R)$ coincides with the set of all nonincreasing concave (convex) functions $\varphi\colon \R\to \R$ with $\varphi(s)+s-m\to 0$ as $s\to \infty$ ($s\to -\infty$) for some constant $m\in \R$ and $\varphi(s)\to 0$ as $s\to -\infty$ ($s\to \infty$). This follows from the observation that, for a concave (convex) function $\ps$, $\ps(s)\to 0$ as $s\to \pm\infty$ implies that $\ps'(s)\to 0$ as $s\to \pm \infty$, where $\ps'$ denotes the right-continuous version of the derivative of $\ps$. Therefore, $-\varphi'$ is the distribution (survival) function $\mu_0$ of a probability measure $\mu\in \PP_1(\R)$ with $\int_\R x\, {\rm d}\mu(x)=m$. 
  In particular, the partial orders $\leq_{\rm icv}$ and $\leq_{\rm icx}$ induce lattice structures on $\PP_1(\R)$ via
  \[
   \qquad \big(\mu\wedge_{\rm icv}\nu\big)(s):=\mu_{1,-}(s)\wedge\nu_{1,-}(s),\qquad \big(\mu\vee_{\rm icx}\nu\big)(s):=\mu_{1,+}(s)\vee\nu_{1,+}(s) 
  \]
  and
  \begin{align*}
   \qquad \qquad \big(\mu\vee_{\rm icv}\nu\big)(s)&:=\inf\{\lambda s+a\, |\, \lambda,a\in \R\text{ with }\la t+a\geq \mu_{1,-}(t)\vee\nu_{1,-}(t)\text{ for all }t\in \R \},\\
   \big(\mu\wedge_{\rm icx}\nu\big)(s)&:=\sup\{\lambda s+a\, |\, \lambda,a\in \R\text{ with }\la t+a\leq \mu_{1,+}(t)\wedge\nu_{1,+}(t)\text{ for all }t\in \R \}
  \end{align*}
  for all $s\in \R$. That is, the minimum (maximum) of two distributions is the poinwise minimum (maximum) of their integrated distribution (survival) functions and the maximum (minimum) of two distributions is the upper concave (lower convex) envelope of the pointwise maximum (minimum) of their integrated survival functions. Recall that, for a function $F\colon \R\to [0,\infty)$, its upper concave envelope $F^*$ and its lower convex envelope $F_*$ are given by
  \begin{align*}
   \qquad F^*(s)&:=\inf\{\lambda s+a\, |\, \lambda,a\in \R\text{ with }\la t+a\geq F(t)\text{ for all }t\in \R \}\quad \text{and}\\
   \qquad F_*(s)&:=\sup\{\lambda s+a\, |\, \lambda,a\in \R\text{ with }\la t+a\leq F(t)\text{ for all }t\in \R \}\quad \text{for }s\in \R.
  \end{align*}
  Since the convergence in the Wasserstein-$1$ topology implies the pointwise convergence of the integrated distribution (survival) functions (see e.g.~\cite[Theorem 6.9]{MR2459454}), the lattice operations $(\mu,\nu)\mapsto \mu \vee_{\rm icx} \nu $ and $(\mu,\nu)\mapsto \mu \wedge_{\rm icx} \nu $ are continuous $\mathcal P_1(\R)\times \mathcal P_1(\R)\to \mathcal P_1(\R)$. Moreover, this implies that the Wasserstein-$1$ topology is finer than both interval topologies, since closed intervals are closed w.r.t. the Wasserstein-$1$ metric.
  \item[b)] Let $(\mu^n)_{n\in \N}\in \mathcal P_1(\R)$ be a sequence, which is bounded above w.r.t.~$\leq_{\rm icv}$ ($\leq_{\rm icx}$), and let
  \[
   \nu^n:=\mu^1\vee_{\rm icv}\ldots \vee_{\rm icv} \mu^n\quad \big(\nu^n:=\mu^1\vee_{\rm icx}\ldots \vee_{\rm icx} \mu^n\big)\quad \text{for all }n\in \N.
  \]
  Then, the supremum of $(\mu^n)_{n\in \N}$ w.r.t. $\leq_{\rm icv}$ ($\leq_{\rm icx}$) exists, and its integrated distribution (survival) function is exactly the pointwise limit, as $n\to \infty$, of the integrated distribution (survival) functions of $(\nu^n)_{n\in \N}$.
  \item[c)] Let $(\mu^n)_{n\in \N}\in \mathcal P_1(\R)$ be a sequence, which is bounded below w.r.t.~$\leq_{\rm icv}$ ($\leq_{\rm icx}$), and let
  \[
   \nu^n:=\mu^1\wedge_{\rm icv}\ldots \wedge_{\rm icv} \mu^n \quad \big(\nu^n:=\mu^1\wedge_{\rm icx}\ldots \wedge_{\rm icx}\mu^n\big)\quad \text{for all }n\in \N.
  \]
 Then, the infimum of $(\mu^n)_{n\in \N}$ w.r.t. $\leq_{\rm icv}$ ($\leq_{\rm icx}$) exists, and its integrated distribution (survival) function is exactly the pointwise limit, as $n\to \infty$, of the integrated distribution (survival) functions of $(\nu^n)_{n\in \N}$.
  %\item[d)] Combining the previous remarks, leads to the following insight: If $(\mu^n)_{n\in \N}\subset \mathcal P_1(\R)$ is a bounded and nondecreasing or nonincreasing sequence, then $(\mu^n)_{n\in \N}$ converges in the Wasserstein-$1$ metric to its supremum or infimum w.r.t.~$\leq_{\rm icx}$, respectively.
  %Let 
 \end{enumerate}
\end{remark}

Before we state the first main result of this subsection, we would like briefly discuss some connections between uniform integrability, and boundedness w.r.t. $\leq_{\rm icv}$ and $\leq_{\rm icx}$.   
\begin{remark}\label{convsupwass}\
 \begin{enumerate}
  \item[a)] Let $K\subset \PP_1(\R)$ and $\underline\mu,\overline\mu\in \PP_1(\R)$ with $\underline\mu\leq_{\rm icv}\mu\leq_{\rm icx}\overline \mu$ for all $\mu\in K$. Then, the identity is u.i. for $K$. In fact,
  \[
  \qquad\sup_{\mu\in K} \int_{2s}^\infty x\, {\rm d}\mu (x)\leq \sup_{\mu\in K} 2\int_{s}^\infty x-s\, {\rm d}\mu (x)\leq 2\int_{s}^\infty x-s\, {\rm d}\overline\mu (x)\to 0\quad \text{as }s\to \infty
  \]
  and
  \[
   \qquad\quad\sup_{\mu\in K} \int_{-\infty}^{2s} -x\, {\rm d}\mu (x)\leq \sup_{\mu\in K} 2\int_{-\infty}^s s-x\, {\rm d}\mu (x)\leq 2\int_{-\infty}^s s-x\, {\rm d}\underline\mu (x)\to 0\quad \text{as }s\to -\infty.
  \]
  \item[b)] Let $(\mu^n)_{n\in \N}\subset \PP_1(\R)$ with $\underline\mu\leq_{\rm icv}\mu^n\leq_{\rm icx} \overline \mu$ and $\mu^n\leq_{\rm icv} \mu^{n+1}$ ($\mu^n\leq_{\rm icx} \mu^{n+1}$) for all $n\in \N$. Then, by part a), $(\mu^n)_{n\in \N}$ is u.i. Therefore, every subsequence of $(\mu^n)_{n\in \N}$ has a Wasserstein-$1$ convergent subsequence (cf. \cite[Theorem 6.9]{MR2459454}). Let $\mu\in \PP_1(\R)$ denote the least upper bound of $(\mu^n)_{n\in \N}$ w.r.t.~$\leq_{\rm icv}$ ($\leq_{\rm icx}$), $(\nu^n)_{n\in \N}$ be a convergent subsequence of $(\mu^n)_{n\in \N}$ and $\nu:=\lim_{n\to \infty}\nu^n$. Then, $\nu^n\leq_{\rm icv} \mu$ ($\nu^n\leq_{\rm icx}\mu$) and therefore, $\nu\leq_{\rm icv} \mu$ ($\nu\leq_{\rm icx} \mu$) by Remark \ref{minimaandmaxima2ndorder} a). On the other hand, for all $n\in \N$, there exists some $N\in \N$ with $\nu^k\geq_{\rm icv} \mu^n$ ($\nu^k\geq_{\rm icx} \mu^n$) for all $k\geq N$. Therefore, by Remark \ref{minimaandmaxima2ndorder} a), $\nu\geq_{\rm icv} \mu^n$ ($\nu\geq_{\rm icx} \mu^n$) for all $n\in \N$, i.e. $\nu\geq_{\rm icv}\mu$ ($\nu\geq_{\rm icx}\mu$). This shows that $\nu=\mu$. Recall that a sequence converges if and only if every subsequence has a convergent subsequence with the same limit. We have therefore shown that $(\mu^n)_{n\in \N}$ converges to its supremum w.r.t.~$\leq_{\rm icv}$ ($\leq_{\rm icx}$) in the Wasserstein-$1$ topology.
  \item[c)] Notice that the $\leq_{\rm icv}$-boundedness from below ($\leq_{\rm icx}$-boundedness from above) in part a) and b) cannot be replaced by $\leq_{\rm icx}$-boundedness from below ($\leq_{\rm icv}$-boundedness from above). Recall that every set $K\subset \PP_1(\R)$ with $\sup_{\mu\in K}\int_\R |x|\, {\rm d}\mu(x)<\infty$ is $\leq_{\rm icv}$-bounded from above ($\leq_{\rm icx}$-bounded from below). This follows from the fact that, by Jensen's inequality, $\mu\leq_{\rm icv} \de_a$ ($\de_a\leq_{\rm icx} \mu$) for all $a\in \R$ and $\mu\in \PP_1(\R)$ with $a\geq \int_{\R}x\, {\rm d}\mu(x)$ ($a\leq \int_{\R}x\, {\rm d}\mu(x)$). However, the uniform boundedness of first moments is not sufficient in order to achieve uniform integrability. We would further like to point out that $\underline\mu\leq_{\rm icv} \mu$ ($\mu\leq_{\rm icx} \overline\mu$) for all $\mu\in K$ and $\int_\R x\, {\rm d}\mu(x)=\int_\R x\, {\rm d}\nu(x)$ for all $\mu,\nu\in K$ implies $\mu\leq_{\rm icx} \underline\mu$ ($\overline\mu\leq_{\rm icv} \mu$) for $\mu\in K$ and thus the uniform integrability of the identity for $K$.
 \end{enumerate}
\end{remark}

We consider a $\sigma$-finite measure space $(S,\mathcal S,\pi)$.~We denote the Borel $\sigma$-algebra w.r.t.~the Wasserstein-$1$ topology by $\mathcal B\big(\PP_1(\R)\big)$ and the lattice of all equivalence classes of $\mathcal S$-$\mathcal B\big(\PP_1(\R)\big)$-measurable functions $S\to \mathcal P_1(\R)$ by $L^0\big(\PP_1(\R)\big)=L^0\big(S,\mathcal S,\pi;\PP_1(\R)\big)$. An arbitrary element $\mu$ of $L^0\big(\PP_1(\R)\big)$ will again be denoted in the form $\mu=(\mu_t)_{t\in S}$. On $L^0\big(\PP_1(\R)\big)$ we consider the order relations $\leq_{L_{\rm icv}^0}$ and $\leq_{L_{\rm icx}^0}$ given by $\mu \leq_{L_{\rm icv}^0} \nu$ if and only if $\mu_t \leq_{\rm icv} \nu_t$ for $\pi$-a.a.\ $t\in S$ and $\mu \leq_{L_{\rm icx}^0} \nu$ if and only if $\mu_t \leq_{\rm icx} \nu_t$ for $\pi$-a.a.\ $t\in S$, respectively.
%The set $L^0\big(\PP_1(\R)\big)$ together with the partial order $\leq_{L_{\rm icv}^0}$ ($\leq_{L_{\rm icx}^0}$) will be denoted by $L^0_{\rm icv}$ ($L^0_{\rm icx}$).

\begin{theorem}\label{thm.completeness2}\
 \begin{enumerate}
  \item[a)] For every fixed $\nu\in L^0\big(\PP_1(\R)\big)$, the lattice $\big\{\mu\in L^0\big(\PP_1(\R)\big)\, \big|\, \mu\leq_{\rm icx}\nu\; (\mu\geq_{\rm icv}\nu)\big\}$ is Dedekind super complete (cf. Definition \ref{def.complete}).
  \item[b)] Let $M\subset L^0\big(\PP_1(\R)\big)$ be nonempty, $\leq_{L^0_{\rm icv}}$-bounded below and $\leq_{L^0_{\rm icx}}$-bounded above. If $M$ is directed upwards/downwards (cf. Definition \ref{def.complete}) w.r.t.~$\leq_{L^0_{\rm icv}}$ ($\leq_{L^0_{\rm icx}}$), then there exists a nondecreasing/nonincreasing sequence $(\mu^n)_{n\in \N}\subset M$ with $\mu^n\to \mu$ $\pi$-a.e.~in the Wasserstein-$1$ topology as $n\to \infty$, where $\mu\in L^0\big(\PP_1(\R)\big)$ is the supremum/infimum of $M$ w.r.t.~$\leq_{L^0_{\rm icv}}$ ($\leq_{L^0_{\rm icx}}$).
 \end{enumerate}
\end{theorem}

\begin{proof}
 Since every $\sigma$-finite measure can be transformed into a finite measure without changing the null-sets, we may w.l.o.g. assume that $\pi$ is finite. By Remark \ref{minimaandmaxima2ndorder} and Remark \ref{convsupwass}, the set $\big\{\mu\in L^0\big(\PP_1(\R)\big)\, \big|\, \mu\leq_{L^0_{\rm icx}}\nu\; (\mu\geq_{L^0_{\rm icv}}\nu)\big\}$ together with $\leq_{L^0_{\rm icv}}$ ($\leq_{L^0_{\rm icx}}$) is Dedekind $\si$-complete. Let $\Phi\colon \R \to \R$ be defined by
 \[
 \Phi^\pm (x):=\frac{1}{\sqrt{2\pi}}\int_\R (x-s)^\pm e^{-s^2/2}\, {\rm d}s \quad \text{for all }x\in \R.                                                                                                                                                                                                                                                                                                                                                                                                                                                                                                                                                            \]
 The map $S\to \R, \; t\mapsto \int_\R\Phi^\pm (x)\, {\rm d}\mu_t(x)$ is $\mathcal S$-$\BB(\R)$-measurable for every $\mu\in L^0\big(\PP_1(\R)\big)$, since $\Phi^\pm$ is $1$-Lipschitz and thus induces a continuous (w.r.t.~the Wasserstein-$1$ topology) linear functional $\mathcal P_1(\R)\to \R$. Hence, $$F^\pm \colon L^0\big(\PP_1(\R)\big)\to \R,\quad \mu\mapsto \int_\R\Phi^\pm(x)\, {\rm d}\mu_t(x)\,{\rm d}\pi(t)$$ is well-defined. Moreover, $F^-$ ($F^+$) is strictly increasing w.r.t.~$\leq_{\rm icv}$ ($\leq_{\rm icx}$), cf. Definition \ref{def:increasing}, since $\Phi$ is nondecreasing and strictly concave (convex), see e.g.~\cite[Theorem 4.A.49]{MR2265633}. The assertions now follow from Lemma \ref{lem.auxres}, Remark \ref{rem.auxres} and Remark \ref{convsupwass} b).
\end{proof}

Again, in the case where $S$ is a singleton and $\pi(S)>0$, we obtain the following proposition.

\begin{proposition}\label{cor.completeness2}\
 \begin{enumerate}
  \item[a)] The lattice $\PP_1(\R)$, endowed with the lattice order $\leq_{\rm icv}$ ($\leq_{\rm icx}$), is Dedekind super complete.
  \item[b)] Let $M\subset \PP_1(\R)$ be nonempty, $\leq_{\rm icv}$-bounded below and $\leq_{\rm icx}$-bounded above. If $M$ is directed upwards/downwards (cf. Definition \ref{def.complete}) w.r.t.~$\leq_{\rm icv}$ ($\leq_{\rm icx}$), then there exists a nondecreasing/nonincreasing sequence $(\mu^n)_{n\in \N}\subset M$ with $\mu^n\to \mu$ in the Wasserstein-$1$ topology as $n\to \infty$, where $\mu\in \PP_1(\R)$ is the supremum/infimum of $M$ w.r.t.~$\leq_{\rm icv}$ ($\leq_{\rm icx}$).
 \end{enumerate}
\end{proposition}

\begin{proof}
 By Remark \ref{minimaandmaxima2ndorder} b) and c), $\PP_1(\R)$ together with $\leq_{\rm icv}$ ($\leq_{\rm icx}$) is Dedekind $\sigma$-complete. The assertions now follow from the same arguments that were employed in the proof of Theorem \ref{thm.completeness2}.
\end{proof}

\begin{remark}
 For fixed $m\in \R$, let $\PP_{1,m}(\R)$ denote the set of all $\mu\in \PP_1(\R)$ with $\int_\R x\, {\rm d}\mu(x)=m$. Then, $\nu\leq_{\rm icv}\mu$ if and only if $\mu\leq_{\rm icx}\nu$ if and only if $\mu\leq_{\rm cx}\nu$ for all $\mu,\nu\in \PP_{1,m}(\R)$. By Theorem \ref{thm.completeness2}, the lattices $L^0\big(\PP_{1,m}(\R)\big)$, defined in a similar way as $L^0\big(\PP_1(\R)\big)$ and equipped with the order $\leq_{L^0_{\rm cx}}:=\leq_{L^0_{\rm icx}}$, as well as $\PP_{1,m}(\R)$ together with the order relation $\leq_{\rm cx}$ are then Dedekind super complete. Moreover, the supremum/infimum of any bounded set $M$, which is directed upwards/downwards can be approximated in the ($\pi$-a.e.)~Wasserstein-$1$ sense by a nondecreasing/nonincreasing sequence in $M$.
\end{remark}

We now turn our focus on characterizing the complete sublattices of $\PP_1(\R)$. We start with the following lemma, which is a combination and generalization of \cite[Theorem 1]{MR2998767} and Lemma \ref{corui1}. It complements the characterizations of uniform integrability from Section \ref{sec.tightness} by two more equivalences.

%\begin{remark}\label{rem.uibdd}
% Let $K\subset \PP_1(\R)$, and assume that the identity is u.i. for $K$, i.e. $\sup_{\mu\in K}\int_\R 1_{[-s,s]^c}(x) |x|\, {\rm d}\mu(x)\to 0$ as $s\to \infty$. Then, $K$ is $\leq_{\rm icx}$-bounded, i.e. there exist $\underline\mu,\overline\mu\in \mathcal P_1(\R)$ with $\underline\mu\leq_{\rm icx} \mu\leq_{\rm icx}\overline\mu$ for all $\mu\in K$. In fact, let $\underline\mu,\overline\mu\colon \mathbb{R} \rightarrow [0,\infty)$ be defined by 
% \begin{align*}
%  \underline\mu(s)&:=\sup\big\{ \la s+a\, \big|\,  \lambda,a\in \R\text{ with }\la t+a\leq \inf_{\mu\in K}\mu(t)\text{ for all }t\in \R\big\},\\
%  \overline\mu(s)&:=\sup_{\mu\in K}\mu(s)
% \end{align*}
% for all $s\in \R$. Then, $\underline \mu$ and $\overline\mu$ are nonincreasing and convex, cf. Remark \ref{minimaandmaxima2ndorder}. Since $K$ is u.i., it follows that
% \[
%  \lim_{s\to -\infty} \underline\mu(s)+s=\inf_{\mu\in K} \int_\R x\, {\rm d}\mu(x)\quad \text{and}\quad \lim_{s\to -\infty} \overline\mu(s)+s=\sup_{\mu\in K} \int_\R x\, {\rm d}\mu(x).
% \]
% Moreover, $\lim_{s\to \infty} F(s)=0$ for $F=\underline\mu,\overline\mu$, which shows that $\underline\mu,\overline\mu\in \mathcal P_1(\R)$. It follows that $\underline\mu=\inf K$ and $\overline\mu=\sup K$ with
% \[
%  \int_\R x\, {\rm d}\underline\mu (x)=\inf_{\mu\in K} \int_\R x\, {\rm d}\mu(x)\quad \text{and}\quad \int_\R x\, {\rm d}\overline \mu (x)=\sup_{\mu\in K} \int_\R x\, {\rm d}\mu(x).
% \]
%\end{remark}

\begin{lemma}\label{equivui}
 Let $K\subset \mathcal P_1(\R)$. Then, the following statements are equivalent:
 \begin{enumerate}
  \item[(i)] The identity is u.i.~for $K$, i.e.~$\sup_{\mu\in K}\int_\R 1_{[-s,s]^c}(x) |x|\, {\rm d}\mu(x)\to 0$ as $s\to \infty$.
  \item[(ii)] There exist $\underline\mu,\overline\mu\in \mathcal P_1(\R)$ with $\underline\mu\leq_{\rm icv} \mu\leq_{\rm icx}\overline\mu$ for all $\mu\in K$.
  \item[(iii)] Every nonempty subset of $K$ has a least upper bound and a greatest lower bound w.r.t. $\leq_{\rm icv}$ and $\leq_{\rm icx}$.
  \item[(iv)] There exists a nondecreasing function $\psi\colon [0,\infty)\to [0,\infty)$ with $\psi(0)=0$, $\frac{\psi(s)}{s}\to \infty$ as $s\to \infty$ and
 \[
  \sup_{\mu\in K}\int_\R \psi(|x|)\,{\rm d}\mu(x)<\infty.
 \]
 \end{enumerate}
 In this case, for $\underline\mu=\inf_{\rm icv} K$ and $\overline\mu=\sup_{\rm icx} K$,
 \begin{equation}\label{minmaxexp}
  \int_\R x\, {\rm d}\underline\mu (x)=\inf_{\mu\in K} \int_\R x\, {\rm d}\mu(x)\quad \text{and}\quad \int_\R x\, {\rm d}\overline \mu (x)=\sup_{\mu\in K} \int_\R x\, {\rm d}\mu(x).
 \end{equation}
 Moreover, the function $\psi$ in (iv) can be chosen to be continuously differentiable, strictly convex and u.i. for $K$. 
\end{lemma}

\begin{proof}
 The equivalence of (i) and (iv) has been discussed in Section \ref{sec.tightness}. By Proposition \ref{cor.completeness2}, (ii) and (iii) are equivalent. Moreover, (ii) implies (i) by Remark \ref{convsupwass}. It remains to show that (i) implies (ii). Assume that $K$ is u.i., and define $\underline\mu_{1,-},\overline\mu_{1,+}\colon \mathbb{R} \rightarrow [0,\infty)$ by 
 \[
  %\underline\mu(s)&:=\sup\big\{ \la s+a\, \big|\,  \lambda,a\in \R\text{ with }\la t+a\leq \inf_{\mu\in K}\mu(t)\text{ for all }t\in \R\big\},
  \underline\mu_{1,-}(s):=\inf_{\mu\in K}\mu_{1,-}(s)\quad \text{and}\quad \overline\mu_{1,+}(s):=\sup_{\mu\in K}\mu_{1,+}(s)
 \]
 for $s\in \R$. Then, $\underline \mu_{1,-}$ is nonincreasing and concave, and $\overline\mu_{1,+}$ is nonincreasing and convex. Since the identity is u.i.~for $K$, it follows that
 \[
  \lim_{s\to \infty} \underline\mu_{1,-}(s)+s=\inf_{\mu\in K} \int_\R x\, {\rm d}\mu(x)\quad \text{and}\quad \lim_{s\to -\infty} \overline\mu_{1,+}(s)+s=\sup_{\mu\in K} \int_\R x\, {\rm d}\mu(x).
 \]
 Moreover, $\lim_{s\to -\infty} \underline\mu_{1,-}(s)=0$ and $\lim_{s\to \infty} \overline\mu_{1,+}(s)=0$, which shows that $\underline\mu_{1,-}$ and $\overline\mu_{1,+}$ give rise to two probability measures $\underline\mu,\overline\mu\in \mathcal P_1(\R)$ with \eqref{minmaxexp}. By definition of $\underline\mu$ and $\overline\mu$, it follows that $\underline\mu=\inf_{\rm icv} K$ and $\overline\mu=\sup_{\rm icx} K$.
\end{proof}

 Notice that the previous lemma is quite interesting in view of Lemma \ref{equivtight}. A combination of these two results yields that uniform integrability of the identity implies tightness and thus $\leq_{\rm st}$-boundedness, which in turn implies $\leq_{\rm icv}$-boundedness and $\leq_{\rm icx}$-boundedness and thus uniform integrability. At first glance, this gives the impression that tightness is equivalent to uniform integrability of the identity and thus seems to be a contradiction. However, the $\leq_{\rm st}$-bounds of tight sets, for which the identity is not uniformly integrable, do not lie in $\PP_1(\R)$, and thus the two lemmas are not contradictory. We are now ready to state the second main result of this subsection.

\begin{theorem}\label{complattice2order}
 Let $K$ be a sublattice of $\mathcal P_1(\R)$. Then, the following statements are equivalent:
 \begin{enumerate}
  \item[(i)] $K$ is compact in the Wasserstein-$1$ topology,
  \item[(ii)] $K$ and $K':=\{\mu'\, |\, \mu\in K\}$ are complete w.r.t.~the lattice order $\leq_{\rm icv}$ ($\leq_{\rm icx}$),
  \item[(iii)] $K$ is complete w.r.t.~the lattice orders $\leq_{\rm icv}$ and $\leq_{\rm icx}$.
 \end{enumerate}
\end{theorem}

\begin{proof}
 If $K=\emptyset$, the statement is trivial. Therefore, we assume that $K$ is nonempty. Since the Wasserstein-$1$ topology is finer than the interval topologies of $\leq_{\rm icv}$ and $\leq_{\rm icx}$, it follows that every Wasserstein-$1$ compact subset of $\mathcal P_1(\R)$ is compact in both interval topologies and thus complete w.r.t.~$\leq_{\rm icv}$ and $\leq_{\rm icx}$. On the other hand, if $K$ is compact in the interval topologies of~$\leq_{\rm icv}$ ($\leq_{\rm icx}$), it is closed in the Wasserstein-$1$ topology. In fact, let $(\mu^n)_{n\in \N}\subset \mathcal P_1(\R)$ with $\mu^n\to \mu\in \mathcal P_1(\R)$ as $n\to \infty$ w.r.t.~the Wasserstein-$1$ metric. Then,
 $$\mu_{1,-}=\sup_{n\in \N}\inf_{k\geq n} \mu_{1,-}^k\in K\qquad \Big(\mu_{1,+}=\inf_{n\in \N}\sup_{k\geq n} \mu_{1,+}^k\in K\Big),$$
 which shows that $\mu\in K$, since $K$ is complete. Here, we used the fact that the Wasserstein-$1$ convergence implies the pointwise convergence of the integrated distribution (survival) functions. This shows that $K$ is Wasserstein-$1$ closed. Moreover, $K$ is $\leq_{\rm icv}$-bounded and $\leq_{\rm icx}$-bounded as it is complete w.r.t.~both partial orders. By Lemma \ref{equivui} (or Remark \ref{convsupwass} a)), it follows that $K$ is u.i., and therefore Wasserstein-$1$ compact, as it is closed in the Wasserstein-$1$ topology (cf. \cite[Theorem 6.9]{MR2459454}). We have thus shown that (i) and (iii) are equivalent. The equivalence of (ii) and (iii) follows from \eqref{icxicv}.
\end{proof}

\begin{corollary}\label{cor.complattice2order}
 Let $K$ be a sublattice of $\mathcal P_1(\R)$ with $\int_\R x\, {\rm d}\mu(x)=\int_\R x \, {\rm d}\nu(x)$ for all $\mu,\nu\in \R$. Then, $K$ is complete w.r.t.~$\leq_{\rm cx}$ if and only if $K$ is compact in the Wasserstein-$1$ topology.
\end{corollary}

 The proof is is a direct consequence of Theorem \ref{complattice2order} together with Remark \ref{convsupwass} c).

%\begin{remark}\label{remintervalui2}
% Consider the situation of the previous lemma. Then, $\mu^{\rm Min}$ and $\mu^{\rm Max}$ can be expressed in terms of $\psi$. Assume that $\psi$ is strictly convex. Let
% \[
% C:=\sup_{\mu\in K}\int_\R |x| d\mu(x) <\infty.
% \]
%We extend $\psi$ to $(-\infty,0)$ by $\psi(s)=0$ for $s<0$. Then, we define
%\[
%\mu^{\rm Min}(s):=\frac{-s}{\psi(-s)}\wedge 1 \quad \text{and} \quad \mu^{\rm Max}(s):=\begin{cases}
%                                                                                      \frac{s}{\psi(s)}\vee -s,& \frac{s}.
%                                                                                     \end{cases}
%\]
%Moreover, the map $ x\mapsto \psi(|x|)^\alpha$ is u.i for $[\mu^{\rm Min},\mu^{\rm Max}]$ for all $0\leq \alpha<1$. In fact,
%\[
% \int_\R \psi(|x|)^\alpha\, d\mu\leq \int_0^\infty \psi(s)^\alpha\, d\mu^{\rm Max}+\int_0^\infty \psi(-s)^\alpha\, d\mu^{\rm Min}.
%\]
%It follows that 
%\[
%  \int_0^\infty \psi(s)^\alpha\, d\mu^{\rm Max}=\int_0^1 \bigg(\frac{C}{1-u}\bigg)^\alpha\, du<\infty
%\]
%and
%\[
% \int_0^\infty \psi(-s)^\alpha\, d\mu^{\rm Min}=\int_0^1 \bigg(\frac{C}{u}\bigg)^\alpha\, du<\infty.
%\]
%By the De La Vall\'{e}e-Poussin Lemma, it follows that $\psi^\alpha$ is u.i. for $[\mu^{\rm Min},\mu^{\rm Max}]$. In particular, if $\psi(s)\geq s^p$ for some $p\in (0,\infty)$, then, $x\mapsto |x|^q$ is u.i. for $[\mu^{\rm Min},\mu^{\rm Max}]$ for all $q\in (0,p)$.
%\end{remark}

\appendix

\section{An auxiliary result}\label{append.auxres}

In this section, we prove an auxiliary result that helps to determine when a Dedekind $\sigma$-complete lattice is super Dedekind complete. We start with the following definitions:

\begin{definition}\label{def.complete}
 Let $L$ be a lattice, i.e.~a partially ordered set (poset) in which every finite nonempty subset has a least upper bound and a greatest lower bound.
 \begin{enumerate}
  \item[a)] We say that $L$ is \textit{Dedekind $\si$-complete} if every countable nonempty subset, that is bounded above or below, has a least upper bound or a greatest lower bound, respectively. We say that $L$ is \textit{Dedekind complete} if every nonempty subset, that is bounded above or below, has a least upper bound or a greatest lower bound, respectively. We say that $L$ is \textit{Dedekind super complete} if every nonempty subset, that is bounded above or below, has a countable subset with the same least upper bound or greatest lower bound, respectively. We say that $L$ is \textit{complete} if every nonempty subset of $L$ has a least upper bound and a greatest lower bound.
  \item[b)] We say that a set $M\subset L$ is \textit{directed upwards} or \textit{directed downwards} if, for all $x,y\in M$, there exists some $z\in M$ with $x\vee y\leq z$ or $x\wedge y\geq z$, respectively.
 \end{enumerate}
\end{definition}

\begin{definition}\label{def:increasing}
 Let $L$ and $R$ be two posets. We say that a map $F\colon L\to R$ is \textit{strictly increasing} if
 \begin{enumerate}
  \item[(i)] $F(x)\leq F(y)$ for all $x,y\in L$ with $x\leq y$,
  \item[(ii)] for all $x,y\in L$ with $x\leq y$ and $F(x)=F(y)$, it follows that $x=y$.
 \end{enumerate}
\end{definition}

The following lemma gives a sufficient and neccessary condition for a Dedekind $\si$-complete lattice to be Dedekind super complete. The proof is a generalized version of the existence proof of the essential supremum for families of real-valued random variables (see e.g. F\"ollmer-Schied \cite[Theorem A.32]{MR2169807}).

\begin{lemma}\label{lem.auxres}

   Let $L$ be a Dedekind $\sigma$-complete lattice. Then, $L$ is Dedekind super complete if and only if there exists a strictly increasing map $F\colon L\to R$ for some Dedekind super complete lattice $R$.
 
\end{lemma}

\begin{proof}
  If $L$ is Dedekind super complete, we may choose $R=L$ and $F$ as the identity. In order to prove the converse implication, let $M\subset L$ be a nonempty subset of $L$, which is bounded above. Then, for every countable nonempty set $\Psi\subset M$, we denote by $x_\Psi:=\sup \Psi\in L$. Let
 \[
  M_0:=\big\{F\big(x_\Psi\big)\, \big|\, \Psi\subset M \text{ nonempty and countable}\big\} \quad \text{and}\quad c:=\sup M_0\in R.
 \]
 Notice that $M_0$ is nonempty, since $M$ is nonempty, and bounded above, since $F$ is nondecreasing (property (i) in Definition \ref{def:increasing}) and $M$ is bounded above. Therefore, $F(a)\in R$ is an upper bound of $M_0$ for every upper bound $a\in L$ of $M$. Since $R$ is Dedekind super complete, there exists a sequence $(\Psi^n)_{n\in \N}$ of countable nonempty subsets of $M$ with
 \[
  \sup_{n\in \N}F(x_{\Psi^n})= c.
 \]
 Then, $\Psi^*:=\bigcup_{n\in \N}\Psi^n$ is a countable nonempty subset of $M$ since it is a countable union of countable nonempty subsets of $M$, and we set $x^*:=x_{\Psi^*}$. Notice that $x_{\Psi^n}\leq x^*$ for all $n\in \N$, and therefore,
 \[
  F\big(x_{\Psi^n}\big)\leq F(x^*)\leq c\quad \text{for all }n\in \N,
 \]
 where the last inequality follows from the fact that $\Psi^*$ is a countable nonempty subset of $M$. Taking the supremum over all $n\in \N$, it follows that $F(x^*)=c$. We now show that $x^*\geq x$ for all $x\in M$. In order to see this, fix some arbitrary $x\in M$, and let $\Psi':=\Psi^*\cup\{x\}$. Then, $\Psi'$ is again a countable nonempty subset of $M$, and we obtain that
 \[
  c=F(x^*)\leq F\big(x_{\Psi'}\big)\leq c. 
 \]
 Since $x^*\leq x_{\Psi'}$ and $F$ is strictly increasing (property (ii) in Definition \ref{def:increasing}), it follows that $x^*=x_{\Psi'}$, which implies that $x\leq x^*$. We have thus shown that $x^*\in L$ is an upper bound of $M$. Now, let $a\in L$ be an upper bound of $M$. Then, $a$ is also an upper bound of $\Psi^*\subset M$, which shows that $a\geq \sup \Psi^*=x^*$. Therefore, $x^*$ is the least upper bound of $M$ and of the countable subset $\Psi^*$ of $M$. Analogously, one shows that $M$ has a countable subset with the same infimum if $M$ is bounded below.
\end{proof}

\begin{remark}\label{rem.auxres}
 Let $M$ be a nonempty subset of a Dedekind super complete lattice $L$. If $M$ is bounded above or below and directed upwards or downwards, then there exists a nondecreasing or nonincreasing sequence $(x^n)_{n\in \N}\subset M$ with $\sup M=\sup_{n\in \N} x^n$ or $\inf M=\inf_{n\in \N} x^n$, respectively. In fact, let $(y^n)_{n\in \N}\subset M$ with $\sup M= \sup_{n\in \N}y^n$. Since $M$ is directed upwards, there exists a sequence $(x^n)_{n\in \N}\subset M$ with $x^{n+1}\geq x^n\geq y^1\vee\dots\vee y^n$ for all $n\in \N$. The sequence $(x^n)_{n\in \N}$ can be constructed recursively by defining $x^1:=y^1$, and by choosing $x^{n+1}\in M$ with $x^{n+1}\geq x^n\vee y^{n+1}$ for all $n\in \N$. Since $(x^n)_{n\in \N}\subset M$, if follows that $\sup_{n\in \N}x^n\leq \sup M$. On the other hand, $y^n\leq x^n$ for all $n\in \N$, and consequently,
 \[
  \sup M=\sup_{n\in \N}y^n\leq \sup_{n\in \N}x^n\leq \sup M.
 \]
 The statement for the infimum follows in an analogous way.
\end{remark}

\section{Proofs of Lemma \ref{firstcrittight} and Lemma \ref{corui1}}\label{appendB}

\begin{proof}[Proof of Lemma \ref{firstcrittight}]
 First, assume that there exists a nondecreasing function $\psi\colon [0,\infty)\to [0,\infty)$ with $\psi(0)=0$, $\psi(s)\to \infty$ as $s\to \infty$ and
 \[
  C:=\sup_{\nu\in K}\int_0^\infty \psi(s)\, {\rm d}\nu(s)<\infty.
 \]
 Then, by Markov's inequality,
 \[
  \sup_{\nu\in K}\nu\big((s,\infty)\big)\leq \frac{1}{\psi(s)}\sup_{\nu\in K}\int_0^\infty \psi(u)\, {\rm d}\nu(u)\leq \frac{C}{\psi(s)}\to 0\quad \text{as }s\to \infty.
 \]
 Now, assume that $K$ is tight, and let $(M^n)_{n\in \N}\subset [0,\infty)$ with $1< M^n\leq M^{n+1}$ and $$\sup_{\nu\in K} \nu\big((M^n,\infty)\big)\leq 2^{-n}$$ for all $n\in \N$. Define $\psi\colon [0,\infty)\to [0,\infty)$ by
 \[
  \psi(s):=\sum_{n\in \N} 1_{(M^n,\infty)}(s)\quad \text{for all }s\geq 0.
 \]
 Then, $\psi$ is nondecreasing with $\psi(0)=0$ and $\psi(s)\to \infty$ as $s\to \infty$. Moreover,
 \[
 \sup_{\nu\in K}\int_0^\infty \psi(s)\, {\rm d}\nu(s)\leq \sum_{n\in \N} \nu\big((M^n,\infty)\big)\leq \sum_{n\in \N}2^{-n}=1.
 \]
 Choosing $\psi$, instead, as the linear interpolation of the points $(0,0)$ and $(M^n,n-1)$ for all $n\in \N$, we obtain that $\psi$ is continuous with $\psi(0)=0$ and
 \[
  \sup_{\nu\in K}\int_0^\infty \psi(s)\, {\rm d}\nu(s)\leq 1,
 \]
 since $\psi\leq \sum_{n\in \N} 1_{(M^n,\infty)}$. If \eqref{convexcrit} is satisfied for some $M\geq 0$, then $\psi$ can, e.g., be chosen as $\psi(s):=(s-M)^+$ for all $s\geq 0$. On the other hand, if $\psi$ is convex, there exist $\alpha>0$ and $\beta\geq 0$ such that $\psi(s)\geq \alpha s-\beta$ for all $s\geq 0$. Then, for $M:=\tfrac{\beta}{\alpha}$,
 \[
  \sup_{\nu\in K}\int_0^\infty (s-M)^+\, {\rm d}\nu(s) =\frac{1}{\alpha}\int_0^\infty (\alpha s-\beta)^+\, {\rm d}\nu(s)\leq \frac{1}{\alpha}\sup_{\nu\in K}\int_0^\infty \psi(s)\, {\rm d}\nu(s)<\infty.
 \]
 Next, assume that \eqref{convexcrit1} is satisfied. Up to now, we saw that the function $\psi$ can be chosen to be strictly increasing on $[M,\infty)$ for $M\geq 0$ sufficiently large. It remains to show that $\psi$ can be chosen to be strictly increasing on $[0,M]$ for all $M>0$. Let $M>0$ and
 \[
  \alpha^0:=\sup_{\nu\in K}\nu\big((M,\infty)\big)<\infty \quad \text{and}\quad \alpha^n:=\sup_{\nu\in K}\nu\big(\big(\tfrac{M}{n+1},\tfrac{M}{n}\big]\big)<\infty\quad \text{for all }n\in\N.
 \]
 Next, we choose a sequence $(c^n)_{n\in \N_0}\subset (0,\infty)$ with $c^{n+1}<c^n$ for all $n\in \N_0$, $c^n\to 0$ as $n\to \infty$ and $\sum_{n\in \N_0} c^n\alpha^n<\infty$. Then,
 \[
  \sup_{\nu\in K} \int_0^\infty c^01_{(M,\infty)}(s)+\sum_{n\in \N} c^n 1_{\big(\tfrac{M}{n+1},\tfrac{M}{n}\big]}(s)\, {\rm d}\nu(s)\leq \sum_{n\in \N_0}c^n\alpha^n<\infty.
 \]
 Let $\psi_M$ denote the linear interpolation of the points $\big(\tfrac{M}{n+1},c^n\big)$ for all $n\in \N_0$ with $\psi_0(0):=0$ and $\psi_M(s):=c_0$ for all $s> M$. Then, $\psi_M$ is strictly increasing on $[0,M]$. Choosing either $\psi(s)=\psi_M(s)+(s-M)^+$ or $\psi=\psi_{M^1}(s)+\eta(s)$ for $s\geq 0$, where $\eta$ is the linear interpolation of the points $(M^n,n-1)$, we see that $\psi$ is strictly increasing. On the other hand, if $\psi$ is strictly increasing, then $\psi(s)>0$ for all $s>0$, and therefore, by Markov's inequality,
 \[
  \sup_{\nu\in K} \nu\big((s,\infty)\big)\leq \frac{1}{\psi(s)}\sup_{\nu\in K}\int_0^\infty \psi(u)\, {\rm d}\nu (u) <\infty\quad \text{for all }s>0.
 \]
\end{proof}

\begin{proof}[Proof of Lemma \ref{corui1}]
 Note that $[0,\infty)\to [0,\infty),\;s\mapsto s$ is uniformly integrable for $K$ if and only if the set $L=\{\nu_1\, |\, \nu\in K\}$ is tight, where, for $\nu\in K$, the measure $\nu_1\colon \BB(\R_+)\to [0,\infty]$ is given by
 \[
  \nu_1(B):=\int_B u\, {\rm d}\nu(u)\quad \text{for all }B\in \BB(\R_+).
 \]
 By the previous lemma, the set $L$ is tight if and only if there exists a nondecreasing function $\eta\colon [0,\infty)\to [0,\infty)$ with $\eta(0)=0$, $\eta(s)\to \infty$ as $s\to \infty$ and
 \[
  \sup_{\sigma\in L}\int_0^\infty \eta(s)\, {\rm d}\sigma(s)<\infty.
 \]
 Defining $\psi(s):=\eta(s)s$ or $\eta(s):=\tfrac{\psi(s)}{s}$ for all $s\geq 0$, respectively, the equivalence of (i) and (ii) follows. By Lemma \ref{firstcrittight}, we may assume that $\eta$ is continuous. Let $\alpha\in (0,1]$ and $\psi_\alpha(s):=\int_0^s\eta(u)^\alpha\, du$ for all $s\geq 0$. Since $\eta$ is continuous and nondecreasing, $\psi_\alpha$ is continuously differentiable, convex and nondecreasing with $\psi_\alpha(0)=0$. Since $\psi_\alpha(s)\leq \eta(s)s$ for all $s\geq 0$, it follows that
 \[
  \sup_{\nu\in K}\int_0^\infty \psi_\alpha(s)\, {\rm d}\nu(s)<\infty.
 \]
 By the transformation theorem, $\tfrac{\psi_\alpha(s)}{s}=\int_0^1\eta(su)^\alpha\, {\rm d}u$ for all $s\geq 0$, which shows that the map $[0,\infty)\to [0,\infty),\; s\mapsto \tfrac{\psi_\alpha(s)}{s}$ is nondecreasing with $\frac{\psi_\alpha(s)}{s}\to \infty$ as $s\to \infty$. Moreover, by H\"older's inequality or Jensen's inequality,
 \[
  \frac{\psi_\alpha(s)}{\psi(s)}=\eta(s)^{-1}\frac{\psi_\alpha(s)}{s}=\eta(s)^{-1}\int_0^1\eta(su)^\alpha\,{\rm d}u\leq \eta(s)^{\alpha-1}.
 \]
 Hence,
 \[
  \sup_{\nu\in K}\int_s^\infty \psi_\alpha(u)\, {\rm d}\nu(u)\leq \eta(s)^{\alpha-1}\sup_{\nu\in K}\int_0^\infty\psi(u)\, {\rm d}\nu(u),
 \]
 which shows that $\psi_\alpha$ is u.i. for $K$ if $\alpha\in (0,1)$. By Lemma \ref{firstcrittight}, the function $\eta$ can, additionally, be chosen to be strictly increasing if and only if \eqref{convexcrit1} is satisfied. In this case, $\psi_\alpha$ can be chosen to be strictly convex and thus strictly increasing. On the other hand, if the function $\psi$ in (ii) is strictly convex, it is strictly increasing, and therefore \eqref{convexcrit1} has to be satisfied by Lemma \ref{firstcrittight}.
\end{proof}

\section*{Acknowledgements}
Financial support through the German Research Foundation via CRC 1283 is gratefully acknowledged. The author thanks Jodi Dianetti and Giorgio Ferrari for many fruitful discussions and helpful comments related to this work.

%%%%%%%%%%%%%%%%%%%%%%%%%%%%%%%%%%

%%%%%%%%%%%%%%%%%%%%%%%%%%%%%%%%%%%%%%%%%%%%%%%%%%%%%%%%%%%%%%%%%%

\bibliographystyle{abbrv}
%\bibliography{Literatur}

\begin{thebibliography}{10}

\bibitem{MR0071400}
R.~M. Baer.
\newblock A characterization theorem for lattices with {H}ausdorff interval
  topology.
\newblock {\em J. Math. Soc. Japan}, 7:177--181, 1955.

\bibitem{MR1700749}
P.~Billingsley.
\newblock {\em Convergence of probability measures}.
\newblock Wiley Series in Probability and Statistics: Probability and
  Statistics. John Wiley \& Sons, Inc., New York, second edition, 1999.
\newblock A Wiley-Interscience Publication.

\bibitem{MR0227053}
G.~Birkhoff.
\newblock {\em Lattice theory}.
\newblock Third edition. American Mathematical Society Colloquium Publications,
  Vol. XXV. American Mathematical Society, Providence, R.I., 1967.

\bibitem{MR2267655}
V.~I. Bogachev.
\newblock {\em Measure theory. {V}ol. {I}, {II}}.
\newblock Springer-Verlag, Berlin, 2007.

\bibitem{MR3412766}
T.~K. Chandra.
\newblock de {L}a {V}all\'{e}e {P}oussin's theorem, uniform integrability,
  tightness and moments.
\newblock {\em Statist. Probab. Lett.}, 107:136--141, 2015.

\bibitem{dffn19}
J.~Dianetti, G.~Ferrari, M.~Fischer, and M.~Nendel.
\newblock Submodular mean field games: Existence and approximation of
  solutions.
\newblock {\em Preprint}, 2019.

\bibitem{MR2169807}
H.~F\"{o}llmer and A.~Schied.
\newblock {\em Stochastic finance}, volume~27 of {\em De Gruyter Studies in
  Mathematics}.
\newblock Walter de Gruyter \& Co., Berlin, extended edition, 2004.
\newblock An introduction in discrete time.

\bibitem{MR6496}
O.~Frink, Jr.
\newblock Topology in lattices.
\newblock {\em Trans. Amer. Math. Soc.}, 51:569--582, 1942.

\bibitem{MR2740082}
T.-C. Hu and A.~Rosalsky.
\newblock A note on the de {L}a {V}all\'{e}e {P}oussin criterion for uniform
  integrability.
\newblock {\em Statist. Probab. Lett.}, 81(1):169--174, 2011.

\bibitem{MR1833858}
R.~P. Kertz and U.~R\"{o}sler.
\newblock Complete lattices of probability measures with applications to
  martingale theory.
\newblock In {\em Game theory, optimal stopping, probability and statistics},
  volume~35 of {\em IMS Lecture Notes Monogr. Ser.}, pages 153--177. Inst.
  Math. Statist., Beachwood, OH, 2000.

\bibitem{MR2998767}
L.~Leskel\"{a} and M.~Vihola.
\newblock Stochastic order characterization of uniform integrability and
  tightness.
\newblock {\em Statist. Probab. Lett.}, 83(1):382--389, 2013.

\bibitem{MR3525602}
H.~Levy.
\newblock {\em Stochastic dominance}.
\newblock Springer, Cham, third edition, 2016.
\newblock Investment decision making under uncertainty.

\bibitem{MR2509253}
S.~Meyn and R.~L. Tweedie.
\newblock {\em Markov chains and stochastic stability}.
\newblock Cambridge University Press, Cambridge, second edition, 2009.
\newblock With a prologue by Peter W. Glynn.

\bibitem{MR2265633}
M.~Shaked and J.~G. Shanthikumar.
\newblock {\em Stochastic orders}.
\newblock Springer Series in Statistics. Springer, New York, 2007.

\bibitem{MR177430}
V.~Strassen.
\newblock The existence of probability measures with given marginals.
\newblock {\em Ann. Math. Statist.}, 36:423--439, 1965.

\bibitem{MR2459454}
C.~Villani.
\newblock {\em Optimal transport}, volume 338 of {\em Grundlehren der
  Mathematischen Wissenschaften [Fundamental Principles of Mathematical
  Sciences]}.
\newblock Springer-Verlag, Berlin, 2009.
\newblock Old and new.

\end{thebibliography}

\end{document}